\documentclass[preprint,12pt,number]{elsarticle}

\usepackage{amssymb}
\usepackage{amsmath}

\usepackage[
  backend=biber,
  style=numeric,
  sorting=nyt,
  giveninits=true,
  terseinits=false,
  maxnames=4,
  minnames=3
]{biblatex}
\DeclareNameAlias{default}{given-family}

\usepackage{amsmath}
\usepackage{amssymb}
\usepackage{amsthm}
\usepackage{graphicx}

\usepackage{float}

\usepackage[margin=1truein]{geometry}

\usepackage{enumitem} 

\usepackage{url}

\usepackage{amsthm}
\theoremstyle{plain} 
\newtheorem{theorem}{\indent\sc Theorem}[section] 

\newtheorem{claim}{\indent\sc Algorithm}

\theoremstyle{definition} 
\newtheorem{definition}[theorem]{\indent\sc Definition}
\newtheorem{remark}[theorem]{\indent\sc Remark}
\newtheorem{example}[theorem]{\indent\sc Example}

\journal{Discrete Applied Mathematics}

\begin{document}

\begin{frontmatter}
\title{Vertex evaluation of multiplex graphs using Forman Curvature} 

\author{Taiki Yamada} 

\affiliation{organization={Interdisciplinary Faculty of Science and Engineering, Shimane University},
            addressline={Nishikawatsu 1060}, 
            city={Matsue},
            postcode={690-8504}, 
            state={Shimane},
            country={Japan}}

\begin{abstract}
The identification of vertices that play a central role in network analysis is a fundamental challenge. Although traditional centrality measures have been extensively employed for this purpose, the increasing complexity of modern networks necessitates the use of sophisticated metrics. The concept of Forman curvature has recently garnered significant attention as a promising approach. We define the Forman curvature for multiplex graphs, which are a category of complex networks characterized by multiple layers of connections between nodes. We then prove the key properties of the Forman curvature in the context of multiplex graphs and show its usefulness in identifying vertices occupying central positions within these networks. Moreover, through a series of comparative experiments with traditional graph features and graph kernels, we demonstrate that the Forman curvature can function as an effective metric for classifying the overall structure of networks.
\end{abstract}




\end{frontmatter}

\section{Introduction}\label{Sec:Introduction}
Graph theory provides essential tools for visualizing and analyzing complex systems in the physical and social sciences, as well as information engineering. To enhance realism, graphs are often augmented with weights assigned to vertices and/or edges, resulting in weighted graphs. Extending fundamental concepts, such as node degree and distance metrics, to weighted graphs introduces complexity, thereby increasing the difficulty of the analysis. Analytical methods specifically designed for doubly-weighted graphs, which feature weights on both vertices and edges, remain relatively scarce. However, even doubly-weighted graphs may not fully capture the intricacies of many real-world situations, where, depending on the event, multiple weights may be assigned to vertices and edges. For example, consider a trading network that models the relationships between companies. Edges represent transactions, but a comprehensive analysis may require multiple associated weights, such as transaction volume, transportation distance, and carbon emissions. Multilayer graphs provide a framework for addressing this complexity. The concept and terminology of multilayer networks have been explored by many researchers over the years. 
For instance, Allard et al. \cite{allard2009heterogeneous} proposed a definition for multilayer or multitype networks based on assigning a type from a set of $M$ types to each node in a set of $N$ nodes. Castelblanco et al. \cite{castelblanco2024combining} employed a multilayer network approach for multidimensional analysis encompassing economic transaction networks and risk propagation networks among stakeholders in megaprojects.
The definitions and implementations of these networks vary considerably. 
In their comprehensive review, Kivela et al. \cite{kivela2014} cataloged 26 distinct names used for related network structures,  noting that while some terms represent the same structure and framework,  others denote specific structural variations.
Kurant et al. \cite{PhysRevLett.96.138701} modeled transportation networks and infrastructure using edge-weighted multilevel graphs, demonstrating advantages over conventional models.
Similarly, Aleta et al. \cite{Aleta2016AMP} analyzed the traffic networks of nine European cities using unweighted multilevel graphs, assessed vulnerability, and identified key nodes via centrality computations. 
In addition, recently, network analysis has been widely applied in the fields of deep learning and time-series analysis. For instance, Yonan et al. \cite{Yonan} proposed an intrusion tendency recognition model that utilized deep learning to leverage network-level features. Adamopoulos et al. \cite{Adamopoulos} developed a time series forecasting model of economic growth that employs an LSTM model incorporating Seq2Seq and attention mechanisms. The efficacy of incorporating network structures and temporal variations into learning models was demonstrated in these studies. Extending this framework facilitates the representation of time-series data as multi-layered networks, wherein temporal transitions are superimposed in layers. Consequently, the integration of multilayer network analysis is anticipated to yield insights that exceed the capabilities of conventional methods. However, this necessitates the development of more sophisticated analytical techniques that can adequately integrate dependencies between layers and diverse information sources.
Therefore, this study proposes a new graph algorithm with applications to transportation networks and time-series data in mind. Typically, when treating transportation networks or time-series data as multilayer networks, the vertex sets of each layer are often identical. Furthermore, the strength of the connections between the layers is thought to depend heavily on the structure of each layer. Consequently, this study introduces a multilayer graph, called a compile graph, based on a multiplex graph.

A key challenge in network analysis is the identification of centrally important vertices. Centrality measures are commonly used for this task. However, the increasing scale and complexity of network data available today necessitate the development of more advanced analytical methods. Specifically, well-established centrality measures tailored for doubly-weighted graphs are lacking, which are the primary focus of this study.
The clustering coefficient is also a powerful concept in network analysis. 
Barrat et al. \cite{Barrat} introduced the clustering coefficient onto edge-weighted graphs and demonstrated its usefulness by applying it to large-scale infrastructure systems. Furthermore, Saucan and Appleboim \cite{CBCforDNA} introduced the metric curvature and clustering coefficient onto doubly-weighted graphs and applied it to DNA Microarray data analysis. Their research not only bridges graph theory and differential geometry but also represents an innovative analytical approach. However, the doubly-weighted graph they employed uses edge weights that are dependent on vertex weights, limiting its applicability. In mathematical modeling involving doubly-weighted graphs, it is customary to define vertex and edge weights independently. For example, when considering transportation networks, vertex weights are often defined as city populations and edge weights as actual distances. It is important to note that, given the independence of these weights from city populations, the actual distances do not depend on the population size.

In parallel with these developments, research on multilayer networks has increasingly emphasized the need for refined centrality notions. For example, Criado et al. \cite{Criado01022012} introduced and demonstrated the utility of a clustering coefficient for unweighted, multilayer graphs. Traditional methods include aggregating multiple layers or combining them into a single graph to utilize the existing centrality measures. Aggregation methods vary, ranging from simple grouping to calculating the average or weighted sum of each layer's scores. 
Aggregation simplifies the structure of complex multilayer networks, which is useful for elucidating the overall network properties. However, according to Cozzo et al. \cite{Cozzo_2015}, this approach carries the risk of losing the distinct structural characteristics and roles of each layer. 
On the other hand, the method of connecting layers is useful because it allows the network structure to be understood without losing information from each layer. However, according to Sole-Ribalta et al. \cite{Albert}, the strength of the inter-layer connections significantly influences centrality. If the definition of these inter-layer connections is arbitrary, it can cause strong layer dependency issues and potentially lead to unstable results. 
This indicates that existing multilayer centrality measures, although valuable, may not always provide a robust characterization of vertex importance in complex settings, such as doubly-weighted graphs. 
Therefore, Ricci curvature, originating from classical Riemannian geometry, has attracted attention as an alternative measure distinct from traditional centrality or clustering metrics. 

The Ricci curvature, which is defined on smooth Riemannian manifolds to quantify local geometry, is not directly applicable to discrete structures such as graphs. However, the discrete Ricci curvature, introduced as an extension of the concept of curvature, has proven valuable for characterizing graph structures (\cite{forman2003bochner, JL, LLY}). Building on this theoretical foundation, research on the application of discrete Ricci curvature has been conducted across various domains.
Several distinct definitions of discrete Ricci curvature exist, with the Ollivier--Ricci curvature (based on optimal transport theory) and Forman--Ricci curvature (based on weight functions) being representative examples.
Ollivier-Ricci curvature uses the $L^1$-Wasserstein distance from optimal transport theory and is related to clustering coefficients.
Saucan and Weber \cite{Saucan} introduced discrete Ricci curvature on multilayer networks and provided an explicit formalism for computing these in practice. They concluded that the discrete Ricci curvature on multilayer networks can be an effective tool for elucidating complex network structures. 
In fact, it has already been applied in various fields on monolayer graphs (\cite{coupette2023ollivierricci,gosztolai2021unfolding,ni2015ricci}).
However, the Ollivier--Ricci curvature cannot be defined on doubly-weighted graphs, and its computation can be expensive owing to the complexity involved.
In contrast, the Forman--Ricci curvature is grounded in combinatorial principles and is instrumental in identifying significant vertices, similar to the role of the Ollivier--Ricci curvature. Moreover, owing to its simpler definition, the Forman--Ricci curvature is generally computationally more efficient than the Ollivier--Ricci curvature.
Thus, the Forman--Ricci curvature has also been applied in various fields,  including the identification of bottlenecks in networks (\cite{de2021using, tee2021enhanced}).
Through a comparative analysis of various networks, Samal et al. \cite{samal2018comparative} demonstrated that both the Ollivier--Ricci and Forman--Ricci curvatures can effectively identify important vertices. Furthermore, they demonstrated that the Ollivier--Ricci curvature may be more suited for networks representing stochastic phenomena, whereas the Forman--Ricci curvature might be more effective for networks reflecting actual distances.  Crucially, the Forman--Ricci curvature can be readily defined and computed for doubly-weighted graphs. Given our goal of developing a versatile analysis method applicable to doubly-weighted multiplex networks, we employed the Forman--Ricci curvature in this study. 
As shown in Remark 3.2 and the results of the sensitivity analysis, the Forman curvature is an effective graph feature. However, since it is defined as a metric on each edge of the graph, it is not optimal as a metric on the graph's vertices. Therefore, this study introduces a comprehensive evaluation as a vertex-based metric using Forman curvature.
It should be noted that a comprehensive evaluation is not intended to replace existing centrality or clustering measures. Rather, it should be viewed as a complementary tool. While conventional measures typically capture a vertex's global influence or prominence within a network, curvature reflects local geometric and structural characteristics, such as the connections and vulnerabilities between vertices. By integrating these perspectives, the Forman--Ricci curvature enhances the explanatory power of conventional metrics and provides a more comprehensive understanding of network organization.
Hence, we will refer to it as the Forman curvature. To support this claim, comparative experiments with traditional graph features and graph kernels are also conducted.

The remainder of this paper is organized as follows.
Section 2 defines the multiplex graph framework used in this study, including relevant terminology. 
Section 3 shows that the Forman curvature on graphs is an effective characteristic for graphs. Then, it introduces the definition of the Forman curvature for multiplex graphs, establishes its theoretical properties, and illustrates its computation method using concrete examples. 
Section 4 introduces a comprehensive evaluation as a metric for vertices and proposes a novel graph algorithm for multiplex graphs. 
Section 5 demonstrates the effectiveness of the graph algorithm by executing it on specific multilayer graphs. In addition, comparative experiments with existing evaluation metrics for classifiers are conducted, proving that comprehensive evaluation is also a crucial metric for classifying graphs.
Finally, Section 6 summarizes our conclusions and discusses the potential directions for future research.

\section{Preliminaries}
A graph is typically represented as a pair $G=(V,E)$, where $V$ is a set of vertices and $E$ is a set of edges connecting pairs of vertices.
In this basic representation, all the edges are considered equivalent. However, in many cases, the edges have distinct values or significance. For example, in a transportation network, an edge may represent an actual distance or transportation cost. To capture this, an edge-weight function $w: E \to \mathbb{R}$ is introduced. A graph equipped with such a function, denoted $V=(G,E,w)$, is called an \textit{edge-weighted graph}. 
Alternatively, assigning values to vertices rather than edges can provide a more accurate representation in specific contexts. For example, in a transaction network, the weight of a vertex can represent the market capitalization or size of a company. Similarly, a vertex-weight function $m:V \to \mathbb{R}$ can be defined as follows: A graph with this function, denoted $G=(V,E,m)$, is called a \textit{vertex-weighted graph}. 
A graph incorporating both types of weights, denoted as $G=(V,E,w,m)$, is termed a \textit{doubly-weighted graph}.

To integrate these graphs and represent them as multiplex graphs, it is necessary to add a new element, that is, a set of “layers,” to this graph representation.
\begin{definition}
A \textit{multiplex graph} is defined by the structure $M = (V_M, E_M, V, \mathcal{L})$, where 
\begin{enumerate}
    \item $V = \left\{u_1, u_2, \cdots, u_n \right\}$ is the set of state vertices, 
    \item $\mathcal{L} = \left\{1, \cdots, L \right\}$ is the set of layers,
    \item $V_M = V \times \mathcal{L}$ is the set of vertices in each layer. Each element ($x^i$) in $V_M$ represents the vertex $x$ in layer $i$, and the subset of state vertices belonging to a specific layer $i$ is denoted by $V_M^i = V \times \left\{i\right\}$.
    \item $E_M \subseteq V_M \times V_M$ is the set of edges, and $E_M$ is typically partitioned into a set of intra-layer edges, defined as $E_A := \left\{(x^i, y^j) \in E_M \mid i = j \right\}$, and a set of inter-layer edges, defined as $E_C := \left\{(x^i, y^j) \in E_M \mid x = y \right\}$. The set of edges in layer $i$ is denoted as $E_M^i$.
\end{enumerate}
Specifically, when $L=1$, $M$ is referred to as a \textit{monolayer graph} and corresponds to the standard graph.
\end{definition}
\begin{remark}
    The name given to a graph formed by layering some graphs varies depending on the characteristics of the original graphs and the integration method used.
    As noted by \cite{kivela2014}, a multiplex graph is a multilayer graph that possesses the same set of vertices in each layer and contains inter-layer edges connecting only the corresponding vertices across different layers. 
    Therefore, this study also adopts the term multiplex graph.
\end{remark}
 \begin{figure}[H]
     \centering
     \includegraphics[scale=0.5]{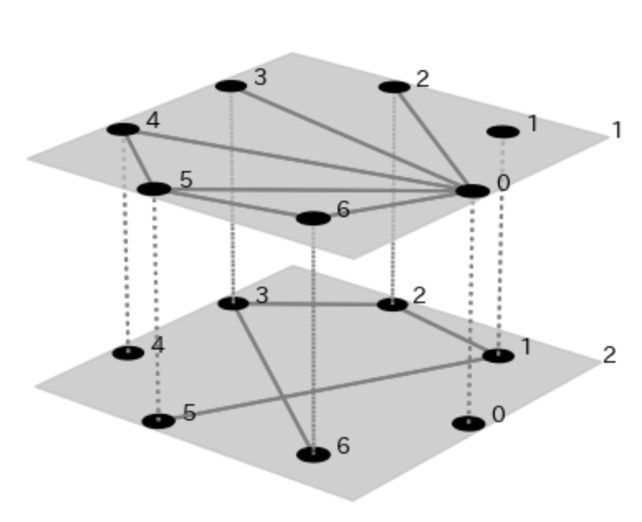} 
     \caption{Example of a multiplex graph $M=(V_M,E_M,V,\mathcal{L})$, where $V=\left\{0,1,2,3,4,5,6 \right\}$, $\mathcal{L}=\left\{1,2\right\}$, and $E_M$ comprises 11 edges.
     }
     \label{multilayer}
 \end{figure}

Similar to monolayer graphs, weight functions can be defined for multiplex graphs. Specifically, $w : E_M \to \mathbb{R}$ and $m : V_M \to \mathbb{R}$ are edge-and vertex-weight functions, respectively. A multiplex graph, denoted by $M=(V_M,E_M,V,\mathcal{L},w,m)$, is called a \textit{doubly-weighted multiplex graph}.

Henceforth, unless stated otherwise, the term “multiplex graph” in this paper refers to a doubly-weighted multiplex graph.

We now adapt the fundamental graph concepts to the multiplex setting.
\begin{definition}
    Let $M = (V_M, E_M, V, \mathcal{L}, w, m)$ be a multiplex graph, where: 
    \begin{enumerate}
    \item For any state vertex $x^i \in V_M$, the \textit{neighborhood} of $x^i$ is defined by
    \begin{eqnarray*}
    \Gamma (x^i) &=& \Gamma_A(x^i) \cup \Gamma_C (x^i)\\
    &:=&\left\{ y^i \in V_M \mid (x^i, y^i) \in E_A \right\} \cup \left\{ x^j \in V_M \mid (x^i, x^j) \in E_C \right\}.
    \end{eqnarray*}
    We also use the notation $\Gamma^{i,j}_C (x) := \Gamma_C(x^i) \setminus \left\{x^j \right\} = \Gamma_C (x^j) \setminus \left\{x^i \right\}$.
    \item For any state vertex $x^i \in V_M$, the \textit{degree} of $x^i$ is defined as
    \begin{eqnarray*}
    \deg (x^i) = \sum_{y^i \in \Gamma (x^i)}w((x^i,y^i)).
    \end{eqnarray*}
    \end{enumerate}
\end{definition}

When constructing a multiplex graph by combining multiple doubly-weighted graphs, the intra-layer weight of each edge of $E_A$ can be directly inherited from the original graph. However, it is important to define the weight of each inter-layer edge. Therefore, we introduce the concept of a compile graph. 
A compile graph is a comprehensive representation that can compile multiple doubly weighted graph information, as well as inter-layer information using specific inter-layer weights.
For this purpose, we first define the following quantity for any state vertex $x^i \in V_M$.
\begin{eqnarray*}
    W(x^i) := 
    \begin{cases}
        \left( \sum_{y^i \in \Gamma_A (x^i)} \cfrac{1}{\sqrt{w((x^i,y^i))}}\right)^{-1},& \text{if}\ \Gamma_A(x^i) \neq \emptyset,\\
        0,& \text{if}\ \Gamma_A(x^i) = \emptyset.
    \end{cases}
\end{eqnarray*}
\begin{definition}
    Let $\left\{G_l = (V, E_l, w_l, m_l) \right\}_{l=1}^L$ be a sequence of doubly-weighted graphs with the same number of vertices. 
    A \textit{compile graph} $CG$ is then defined as the multiplex graph     \begin{eqnarray*}
        V(CG) &=& V,\ \mathcal{L} = \left\{1, \cdots, L \right\},\ V_{CG} = \bigsqcup_l (V(G_l) \times \left\{l \right\}),\\
        E_A &:=& \left\{(x^i,y^j) \in V_{CG} \times V_{CG} \mid i=j,\ (x,y) \in E(G_i)\right\},\\
        E_C &:=& \left\{ (x^i,x^j) \mid x \in V,\ i,j\in \mathcal{L},\ i \neq j \right\},
    \end{eqnarray*}
   where $m(x^i) = m_i(x)$ for any vertex $x^i \in V_{CG}$, $w((x^i,y^i)) = w_i ((x,y))$ for any intra-layer edge $(x^i, y^i) \in E_A$, and
    \begin{eqnarray*}
    w((x^i,x^j)) = \min \left\{ W^2(x^i), W^2(x^j) \right\} 
    \end{eqnarray*}
    for any inter-layer edge $(x^i,x^j) \in E_C$.  
\end{definition}

The following remark provides insight into the technical definition.
\begin{remark}
    If all intra-layer edge weights equal 1, then $W(x^i)=1/|\Gamma_A(x^i)|$. In this case, the smaller the number of edges connected to vertex $x^i$, the larger the value of $W(x^i)$.
    To further investigate the behavior of $W(x^i)$ in more detail when the intra-layer edge weights are not equal to 1, we consider the function $f(x, y) = 1/\sqrt{x} + 1/\sqrt{y}$. If $x>0$, then     \begin{eqnarray*}
        \frac{\partial f}{\partial x} = - \frac{1}{x^{3/2}} < 0.
    \end{eqnarray*}
    Thus, $f(x,y)$ is a monotonically decreasing function when $x>0$ and $y>0$. This is because the partial derivatives of $y$ are obtained in a similar manner.
    As $W(x^i)$ can be expressed as an inverse function of $f(x,y)$,  $W(x^i)$ monotonically increases with respect to each incident intra-layer edge weight. That is, $W(x^i) < W(x^j)$ suggests that the overall weight associated with the intra-layer edges connected to vertex $x^j$ is greater than that connected to vertex $x^i$. 

    One rationale for employing complex weights as inter-layer weights in compile graphs is the practical advantages they confer.
    For example, consider a scenario in which the identification of vulnerable points is crucial for developing a resilient supply chain. According to Sharma et al. \cite{SCVA},  the priority vulnerable points in a supply chain are not temporary weaknesses but rather those that are consistently weak. Therefore, a multiplex graph is constructed using annual supply chain data. By defining the inter-layer weights as the compile graph definition, it becomes possible to efficiently extract these consistently weak locations.

    Moreover, the strategic allocation of these weights between layers enables the efficient computation of discrete curvature (refer to Section 3.2) and plays a pivotal role in the classification of multiplex graphs (refer to Section 5.2).
\end{remark} 
\section{Forman Curvature}
This section defines the Forman curvature, which is the specific type of discrete curvature used in this study.
\subsection{Forman curvature on graphs}
This subsection introduces the definition of the Forman curvature for standard (monolayer) graphs and outlines their key characteristics.
\begin{definition}[Forman \cite{forman2003bochner}]
Let $G=(V,E,w,m)$ be a doubly-weighted graph. The \textit{Forman curvature} of any edge $e = (x, y)$ is defined as
\begin{eqnarray*}
F((x,y)) = 2(m(x) +m(y)) - m(x)\sum_{(x,z) \in E} \sqrt{\frac{w((x,y))}{w((x,z))}} - m(y)\sum_{(y,w) \in E}\sqrt{\frac{w((x,y))}{w((y,w))}}.
\end{eqnarray*}
\end{definition}

\begin{remark}
\label{Forman property}
The properties of the Forman curvature were proven by \cite{samal2018comparative, sreejith2016forman}, and its efficacy as a graph feature was substantiated. The following sentence provides a concise overview of the properties relevant to this study.
\begin{enumerate}
\item The Forman curvature often takes on negative values.
\item The Forman curvature influences large-scale network connectivity. Removing nodes associated with low average Forman curvature can lead to faster network fragmentation than removing nodes based solely on high clustering coefficients. This suggests that nodes with a small Forman curvature are the most important for maintaining network integrity.
\item Unlike some traditional metrics for network structure analysis, such as centrality and clustering coefficient, which are primarily defined for unweighted or only edge-weighted graphs, Forman curvature can be defined for doubly-weighted graphs, making it suitable for the graphs considered in this study.
\item In unweighted graphs, the Forman curvature is negatively correlated with both the centrality and clustering coefficients (see Table 3 in \cite{samal2018comparative}).
\end{enumerate}
\end{remark}

\subsubsection{Sensitivity}
It has been discussed that vertex and edge weights are important for calculating the Forman curvature. Therefore, in this section, an evaluation is conducted of the extent to which the variation in the Forman curvature is dependent on the weight parameters $m(x)$, $m(y)$ and $w(e)$ for any $e \in E$. To facilitate this evaluation, the partial derivative of the Forman curvature with respect to each weight parameter is calculated, yielding the following expression:
\begin{enumerate}
    \item Partial derivative with respect to the vertex weight of $x$ and $y$
        \begin{eqnarray*}
            \begin{cases}
                \frac{\partial F((x,y))}{\partial m(x)} = 1 - \sum_{(x,z) \in E} \sqrt{\frac{w((x,y))}{w((x,z))}},\\
                \frac{\partial F((x,y))}{\partial m(y)} = 1 - \sum_{(y,w) \in E} \sqrt{\frac{w((x,y))}{w((y,w))}}.
            \end{cases}
        \end{eqnarray*}
    \item Partial derivative with respect to the edge weight of $e=(x,y)$
        \begin{eqnarray*}
            \frac{\partial F((x,y))}{\partial w((x,y))} = - \frac{1}{2\sqrt{w((x,y))}}\left(\sum_{(x,z) \in E} \frac{m(x)}{\sqrt{w((x,z))}} + \sum_{(y,w) \in E} \frac{m(y)}{\sqrt{w((y,w))}}\right).
        \end{eqnarray*}
    \item Partial derivative with respect to other edge weights
        \begin{eqnarray*}
        \frac{\partial F((x,y))}{\partial w(e)} =
            \begin{cases}
                 \frac{m(x)\sqrt{w((x,y))}}{2 w^{3/2}(e)},& \text{if}\ e \cap x \neq \emptyset,\\
                 \frac{m(y)\sqrt{w((x,y))}}{2 w^{3/2}(e)},& \text{if}\ e \cap y \neq \emptyset,\\
                 0,& \text{otherwise}.
            \end{cases}
        \end{eqnarray*}
\end{enumerate}
It should be noted that, with the exception of edges characterized by extremely low influence (i.e., $w((x,y))$ is very small), $\sum_{(x,z) \in E} \sqrt{\frac{w((x,y))}{w((x,z))}}$ and $\sum_{(y,w) \in E} \sqrt{\frac{w((x,y))}{w((y,w))}}$ take values of 1 or greater. 

It is evident from these calculations that an increase in the values of $m(x)$ or $m(y)$ results in a decrease in $F((x,y))$. Similarly, an increase in the value of $w((x,y))$ results in a decrease in $F((x,y))$. In contrast, an increase in the weights of the edges adjacent to edge $(x, y)$ results in an increase in $F((x,y))$. Therefore, the higher the importance of edge $(x,y)$ relative to its surroundings, the lower the value of $F((x,y))$. To more precisely clarify the change in $F((x,y))$, we define the dimensionless sensitivity $S$ as follows: For any weight parameter $p \in \left\{m(x), m(y)\right\} \cup \left\{w(e) \mid e \in E \right\}$,
\begin{eqnarray*}
    S_p((x,y)) := \frac{\partial F((x,y))}{\partial p} \frac{p}{F((x,y))}.
\end{eqnarray*}
For Zachary's karate club, we compute $S_p$ and visualize its change. The results are illustrated in Figure \ref{Sensitivity Analysis}, where the edge sensitivity maps highlight the degree of variation between the schemes used.

 \begin{figure}[H]
     \centering
     \includegraphics[scale=0.45]{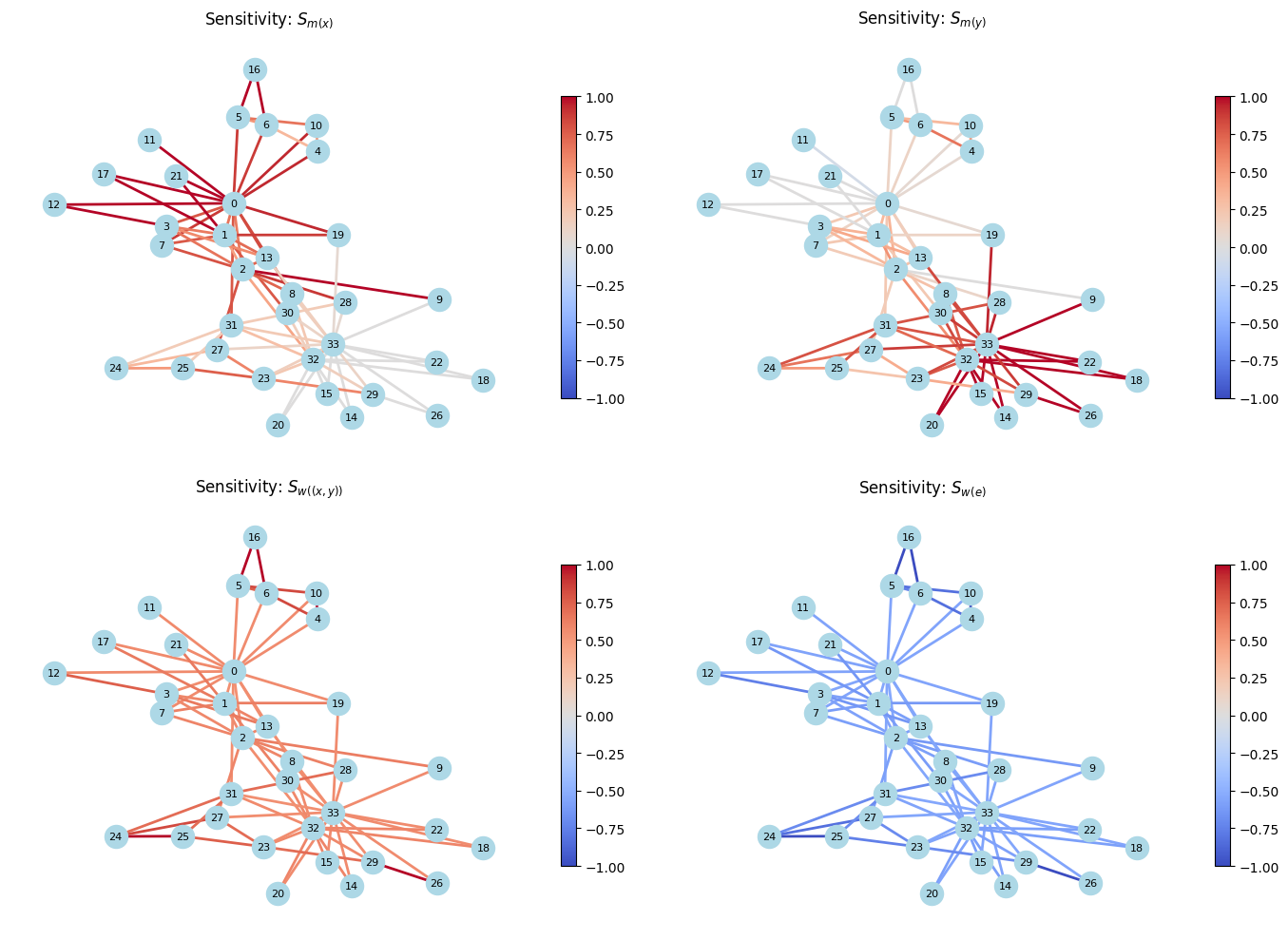} 
     \caption{Sensitivity Analysis for Weighting Schemes}
     \label{Sensitivity Analysis}
 \end{figure}

While the distribution of the Forman curvature exhibited local variations, the spatial patterns of the high- and low-curvature regions across the entire network remained consistent. This finding suggests that the curvature values exhibit relative stability with respect to weighting settings. In this context, the term "stability" is employed to denote the first observation that the numerical fluctuation in the curvature values is limited. The second observation is that the reshuffling of edge and node rankings based on the curvature levels is minimal. The third observation is that the overall network pattern is not reversed by the weighting. Given the fulfillment of all the aforementioned conditions in this analysis, it is concluded that structural interpretations based on the Forman curvature are reliable and do not depend excessively on the selected weighting scheme. 

It is imperative to acknowledge that if the sensitivity analysis reveals substantial alterations in curvature signs or a notable reshuffling of high-importance edge groups due to minor weighting modifications, the curvature interpretation is considered unstable.

\subsubsection{Example}
\label{subsec:ex}
To illustrate the relationship between the Forman curvature and graph structure, we computed it for specific graphs. For simplicity, we consider unweighted graphs ($w=1$ and $m=1$). The Forman curvature is computed as follows:
\begin{eqnarray*}
F((x,y)) = 4 - \deg(x) - \deg(y).
\end{eqnarray*}

\begin{example}
Let $G = K_n$ be a complete graph with $n$ vertices.
\end{example}
For any vertex $x \in V$, $\deg (x) = n-1$. Thus, the Forman curvature of any edge $e=(x,y)$ is 
\begin{eqnarray*}
F((x,y)) = 4 - 2(n-1).
\end{eqnarray*}

\begin{example}
Let $G = C_n$ be a cycle graph with $n$ vertices.
\end{example}
For any vertex $x \in V$, $\deg (x) = 2$. Thus, the Forman curvature of any edge is equal to zero.

\begin{example}
Let $G = T_r$ be a $r$-regular tree graph.
\end{example}
For any vertex $x \in V$, we have $\deg (x) = r$. Thus, the Forman curvature of any edge $e=(x,y)$ is 
\begin{eqnarray*}
F((x,y)) = 4 - 2r.
\end{eqnarray*}

 \begin{figure}[H]
     \centering
     \includegraphics[scale=0.3]{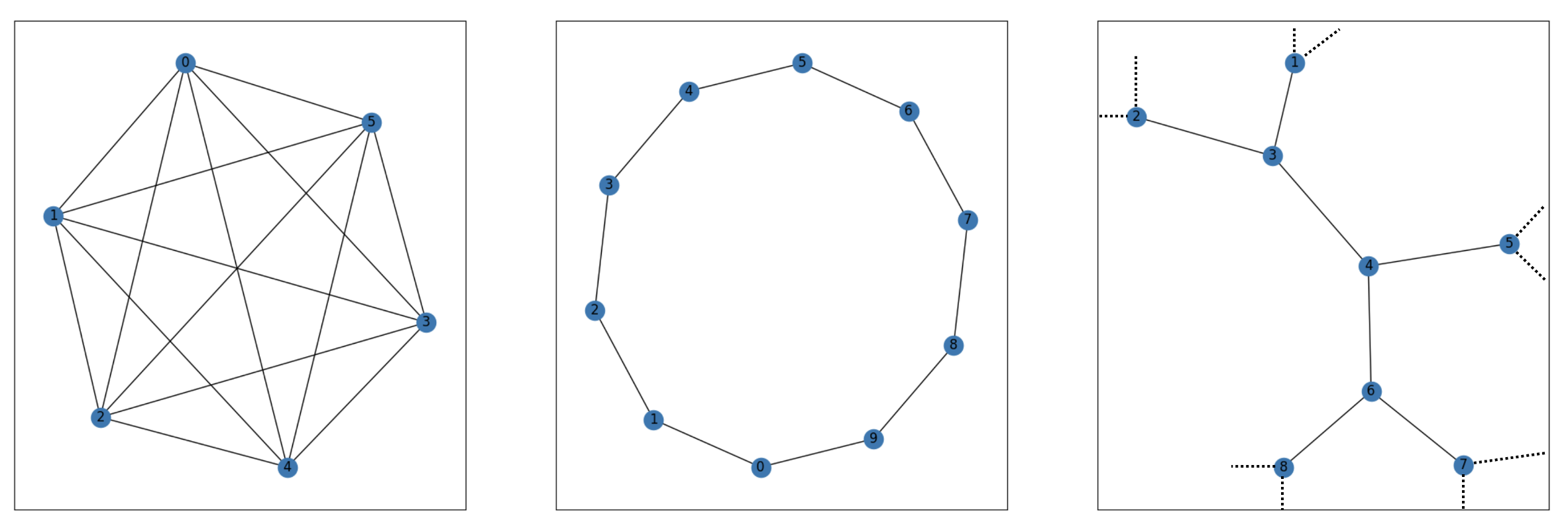} 
     \caption{Example graphs (from left to right: complete, cycle, and tree)}
     \label{graph example}
 \end{figure}

  As suggested by the definition of Forman curvature, the Forman curvature usually has negative values for most of the edges. However, it is important to note that a negative value does not automatically indicate that a node or edge acts as a bridge. In fact, both trees and complete graphs predominantly exhibit negative curvature, despite their substantially different structural roles. Instead, the magnitude of the curvature can be understood as capturing the extent of the "geometric tension" in the local network structure. More precisely, the greater the number of edges connected to an edge, the more negative the value of the Forman curvature. Therefore, the smaller the value of the Forman curvature, the more important the edge is in the network.
  In multilayer settings, edges that link otherwise strongly connected layers often exhibit a strongly negative curvature, reflecting their role in maintaining inter-layer connectivity. Therefore, while negative curvature values are common, relatively low curvature can highlight nodes or edges that are structurally significant for network integrity across layers. This makes curvature a complementary descriptor of centrality, which is sensitive to both intra- and inter-layer connectivity patterns.

\subsection{Forman curvature on compile graphs}
The definition of Forman curvature can be easily extended to multiplex graphs, provided that the vertex and edge weights are defined across all state vertices and edges.
\begin{definition}
    Let $M = (V_M, E_M, V, \mathcal{L}, w, m)$ be a multiplex graph. For any edge $e =(x^i, y^j) \in E_M$, the Forman curvature is
    \begin{eqnarray*}
        F(e) = 2(m(x^i) +m(y^j)) - m(x^i)\sum_{\bar{x}^k \in \Gamma(x^i)} \sqrt{\frac{w(e)}{w(x^i,\bar{x}^k)}} - m(y^j)\sum_{\bar{y}^l \in \Gamma(y^j)}\sqrt{\frac{w(e)}{w(y^j,\bar{y}^l)}}.
    \end{eqnarray*}
    Specifically, for an intra-layer edge $e = (x^i,y^i) \in E_A$,
    \begin{eqnarray*}
        F(e) = F^i((x,y)) - m(x^i)\sum_{x^k \in \Gamma_C(x^i)} \sqrt{\frac{w((x^i,y^i))}{w((x^i,x^k))}} - m(y^j)\sum_{y^k \in \Gamma_C(y^i)}\sqrt{\frac{w((x^i,y^i))}{w((y^i,y^k))}},
    \end{eqnarray*}
    where $F^i$ is the Forman curvature on the $i$-th graph.
\end{definition}

For the specific case of compile graphs, we introduce the following notations:
\begin{eqnarray*}
    \Gamma^{i,j}_{C,-}(x) := \left\{x^k \in \Gamma^{i,j}_C (x)\mid W(x^k) \leq \min \left\{W(x^i), W(x^j) \right\} \right\},\\
    \Gamma^{i,j}_{C,+}(x) := \left\{x^k \in \Gamma^{i,j}_C (x)\mid W(x^k) \geq \max \left\{W(x^i), W(x^j) \right\} \right\}.
\end{eqnarray*}
The Forman curvature of any inter-layer edge can then be simplified as follows:
\begin{definition}
    Let $CG = (V_{CG}, E_{CG}, V, \mathcal{L}, w, m)$ be a compile graph. The Forman curvature of any inter-layer edge $e =(x^i, x^j) \in E_C$ with $W(x^i) \leq W(x^j)$is expressed as
    \begin{eqnarray*}
        F((x^i,x^j)) &=& m(x^i)\left\{1 - W(x^i)\sum_{x^k \in \Gamma_C(x^i)} \left(\frac{1}{W(x^i)} \vee \frac{1}{W(x^k)}\right) \right\}\\
        &+ & m(x^j)\left\{2 - \frac{W(x^i)}{W(x^j)} - W(x^i)\sum_{x^l \in \Gamma_C(x^j)} \left(\frac{1}{W(x^j)} \vee \frac{1}{W(x^l)}\right) \right\} \\
        &=&-m(x^i)|\Gamma_C^{i,j}(x) \setminus \Gamma_{C,-}^{i,j}(x)| + m(x^j)\left(1 - (|\Gamma_{C,+}^{i,j}(x)|+1)\cfrac{W(x^i)}{W(x^j)}\right) \nonumber\\
        &\ & -m(x^i) \sum_{x^k \in \Gamma_{C,-}^{i,j}(x)} \cfrac{W(x^i)}{W(x^k)} - m(x^j)\sum_{x^l \in \Gamma^{i,j}_C(x) \setminus \Gamma_{C,+}^{i,j}(x)}\cfrac{W(x^i)}{W(x^l)}
    \end{eqnarray*}
    where, for real numbers $s$ and $t$, $s \wedge t := \min \left\{s,t\right\}$ and $s \vee t := \max \left\{s,t\right\}$.
\end{definition}

We now establish some properties of the Forman curvature in multiplex graphs.
\begin{theorem}
\label{lower}
    Let $CG = (V_{CG}, E_{CG}, V, \mathcal{L}, w, m)$ be a compile graph. 
    For any inter-layer edge $e =(x^i, x^j) \in E_C$ with $W(x^i) \leq W(x^j)$, we then have
    \begin{eqnarray}
    \label{Formanvalue}
        F((x^i,x^j)) 
        &\geq& -(m(x^i)+m(x^j)) |\Gamma^{i,j}_{C,-} (x)|\max \left\{ \cfrac{W(x^i)}{W(x^l)} \mid x^l \in \Gamma^{i,j}_{C,-}(x)\right\} \nonumber\\
        &\ & - m(x^j)\left(\cfrac{W(x^i)}{W(x^j)} -1 \right)\\
        &\geq& -(m(x^i)+m(x^j)) |\Gamma^{i,j}_{C,-} (x)|\max \left\{ \cfrac{W(x^i)}{W(x^l)} \mid x^l \in \Gamma^{i,j}_{C,-} (x)\right\}.
    \end{eqnarray}
    Equality holds in the first inequality line if and only if $\Gamma^{i,j}_C (x)= \Gamma^{i,j}_{C,-}(x)$. Equality holds in the second inequality line if and only if $W(x^i) = W(x^j)$.
\end{theorem}

\begin{proof}
    Note that for any subset $A$ of $X$, if the number of elements in $A$ decreases by one, the number of elements in $X\setminus A$ increases by one.

    If there exists a vertex $x^k \in V_{CG}$ such that $x^k \notin \Gamma_{C,-}^{i,j}(x)$, we consider another compile graph $\bar{CG}$ where only the value of $W(x^k)$ is modified to ensure that $x^k$ is included in $\Gamma_{C,-}^{i,j}(x)$. Let the modified value of $W(x^k)$ be $\bar{W}(x^k)$. In this case, the difference in the Forman curvature is calculated as follows:
    \begin{eqnarray*}
        F^{CG}((x^i,x^j)) - F^{\bar{CG}}((x^i,x^j)) = m(x^i) \left( \cfrac{W(x^i)}{\bar{W}(x^k)} - 1\right) \geq 0,
    \end{eqnarray*}
    which implies $F^{CG}((x^i,x^j)) \geq F^{\bar{CG}}((x^i,x^j))$. Thus, the larger the cardinality of $\Gamma_{C,-}^{i,j}(x)$, the smaller is the Forman curvature.

    Conversely, if there exists a vertex $x^k \in V_{CG}$ such that $x^k \notin \Gamma_{C,+}^{i,j}(x)$, we consider another compile graph $\hat{CG}$ where only the value of $W(x^k)$ is modified to ensure that $x^k$ is included in $\Gamma_{C,+}^{i,j}(x)$. Let the modified value of $W(x^k)$ be $\hat{W}(x^k)$. In this case, the difference in Forman curvature is calculated as follows:
    \begin{eqnarray*}
        F^{CG}((x^i,x^j)) - F^{\hat{CG}}((x^i,x^j)) = m(x^j) \left( \cfrac{W(x^i)}{W(x^j)} - \cfrac{W(x^i)}{W(x^k)}\right) \leq 0,
    \end{eqnarray*}
    which implies $F^{CG}((x^i,x^j)) \leq F^{\hat{CG}}((x^i,x^j))$. Thus, the larger the cardinality of $\Gamma_{C,+}^{i,j}(x)$, the larger  the Forman curvature.

    Based on these results, the minimum value of the Forman curvature occurs when the cardinality of $\Gamma_{C,-}^{i,j}(x)$ coincides with that of $\Gamma_C^{i,j}(x)$. The minimum value is calculated as follows:
    \begin{eqnarray*}
        F((x^i,x^j)) 
        &\geq& -(m(x^i)+m(x^j)) |\Gamma^{i,j}_{C,-} (x)|\max \left\{ \cfrac{W(x^i)}{W(x^l)} \mid x^l \in \Gamma^{i,j}_{C,-}(x)\right\} \nonumber\\
        &\ & - m(x^j)\left(\cfrac{W(x^i)}{W(x^j)} -1 \right).
    \end{eqnarray*}

    Next, we examine the maximum value of $W(x^i)/W(x^j)$. Given the condition that $W(x^i) \leq W(x^j)$, the maximum value of $W(x^i)/W(x^j)$ is 1 when $W(x^i)=W(x^j)$. Thus, the lower bound of the Forman curvature was proven.
\end{proof}

\begin{theorem}
\label{upper}
    Let $CG = (V_{CG}, E_{CG}, V, \mathcal{L}, w, m)$ be a compile graph. 
    Then, for any inter-layer edge $e =(x^i, x^j) \in E_C$ with $W(x^i) \leq W(x^j)$, the Forman curvature of $e$ is represented as follows:
    \begin{eqnarray*}
        F((x^i,x^j)) \leq -|\Gamma^{i,j}_C(x)|m(x^i) + m(x^j)\left(1 - (|\Gamma_C^{i,j}(x)|+1)\cfrac{W(x^i)}{W(x^j)}\right).
    \end{eqnarray*}
    The condition for equality is $\Gamma^{i,j}_C (x)= \Gamma^{i,j}_{C,+}(x)$.
\end{theorem}

\begin{proof}
    In the proof of Theorem \ref{lower}, we established the following relationships:
    \begin{enumerate}
        \item The smaller the cardinality of $\Gamma_{C,-}^{i,j}(x)$, the larger the Forman curvature.
        \item The larger the cardinality of $\Gamma_{C,+}^{i,j}(x)$, the larger the Forman curvature.
    \end{enumerate}
    Thus, the Forman curvature attains its maximum value when $\Gamma^{i,j}_C (x)= \Gamma^{i,j}_{C,+}(x)$. The maximum value is computed as follows:
    \begin{eqnarray*}
        F((x^i,x^j)) \leq -|\Gamma^{i,j}_C(x)|m(x^i) + m(x^j)\left(1 - (|\Gamma_C^{i,j}(x)|+1)\cfrac{W(x^i)}{W(x^j)}\right).
    \end{eqnarray*}
    This completes the proof.
\end{proof}

\subsubsection{Example}
To illustrate further, we compile graphs with two or three layers and calculate the Forman curvature of the representative inter-layer and intra-layer edges.
\begin{example}
    Let $CG = (V_{CG}, E_{CG}, V, \mathcal{L}, w, m)$ be a compile graph with $L=2$. 
\end{example}
The Forman curvature of any inter-layer edge $e =(x^1, x^2) \in E_C$ with $W(x^1) \leq W(x^2)$ is calculated as
    \begin{eqnarray*}
        F((x^1,x^2))= m(x^2)\left(1 - \cfrac{W(x^1)}{W(x^2)}\right) \geq 0.
    \end{eqnarray*}
Conversely, the Forman curvature of any intra-layer edge $e = (x^1, y^1) \in E_A$ is given by
    \begin{eqnarray*}
        F(e) = F^1((x,y)) - m(x^1)\frac{\sqrt{w(e)}}{\min \left\{W(x^1), W(x^2) \right\})} - m(y^1)\cfrac{\sqrt{w(e)}}{\min \left\{W(y^1), W(y^2) \right\})}.
    \end{eqnarray*}

\begin{example}
    Let $CG = (V_{CG}, E_{CG}, V, \mathcal{L}, w, m)$ be a compile graph with $L=3$. 
\end{example}
    The Forman curvature of any inter-layer edge $e =(x^1, x^2) \in E_C$ with $W(x^1) \leq W(x^2)$ is then expressed as 
    \begin{eqnarray*}
        &\ &F(e)= -m(x^1) \left(1 \vee \cfrac{W(x^1)}{W(x^3)} \right) + m(x^2)\left(1 - \cfrac{W(x^1)}{W(x^2)} - \cfrac{W(x^1)}{W(x^2)}\vee \cfrac{W(x^1)}{W(x^3)} \right)\\
        &=& 
        \begin{cases}
            -m(x^1) \cfrac{W(x^1)}{W(x^3)} + m(x^2)(1 - \cfrac{W(x^1)}{W(x^2)} - \cfrac{W(x^1)}{W(x^3)}),& \text{if}\ W(x^3) \leq W(x^1)\leq W(x^2),\\
            -m(x^1) + m(x^2)(1 - \cfrac{W(x^1)}{W(x^2)} - \cfrac{W(x^1)}{W(x^3)}),& \text{if}\ W(x^1)<W(x^3)<W(x^2),\\
            -m(x^1) + m(x^2)(1 - 2\cfrac{W(x^1)}{W(x^2)}),& \text{if}\ W(x^1) \leq W(x^2) \leq W(x^3).
        \end{cases}
\end{eqnarray*}
Conversely, the Forman curvature of any intra-layer edge $e = (x^1, y^1) \in E_A$ is given by
    \begin{eqnarray*}
        F(e) &=& F^1((x,y)) - m(x^1)\frac{\sqrt{w(e)}}{\min \left\{W(x^1), W(x^2) \right\})} - m(x^1)\frac{\sqrt{w(e)}}{\min \left\{W(x^1), W(x^3) \right\})}\\
        &\ & - m(y^1)\cfrac{\sqrt{w(e)}}{\min \left\{W(y^1), W(y^2) \right\})}- m(y^1)\cfrac{\sqrt{w(e)}}{\min \left\{W(y^1), W(y^3) \right\})}.
    \end{eqnarray*}
These examples demonstrate that the Forman curvature of an inter-layer edge $e=(x^i, x^j)$ in a compile graph depends on whether each $W(x^k)\in \Gamma_C^i,j(x)$ belongs to $\Gamma_{C,+}^i,j(x)$ or $\Gamma_{C,-}^i,j(x)$. Additionally, the Forman curvature of an intra-layer edge is highly influenced by the Forman curvature of the $i$-graph and can serve as a measure of edge importance within the $i$-graph, as well as the Forman curvature of the graph.

Various discrete curvature notions have been studied in single-layer networks, such as the Ollivier–Ricci curvature and Bakry–Émery Ricci curvature. However, their direct extension to multilayer networks remains nontrivial. The main difficulty lies in defining consistent probability measures or diffusion processes across layers, which is a prerequisite for transportation-based definitions, such as the Ollivier–Ricci. In contrast, the Forman curvature is locally defined using only node degrees and edge weights, allowing a straightforward adaptation to multilayer settings without requiring global consistency assumptions.
Therefore, while a direct comparison with other curvatures in multilayer networks is not currently possible, we highlight the differences between them in the single-layer cases. For instance, the Ollivier–Ricci curvature often aligns with the notions of robustness under optimal transport, whereas the Forman curvature provides a more combinatorial, locally computable indicator. This motivates our choice of the Forman curvature as a practical and interpretable tool for multilayer analysis.

\section{Methodology}
If the primary goal of analyzing a multilayer network is to identify vertices that maintain consistent importance across all layers, calculating the importance within each layer separately may suffice. This approach reduces the need for a dedicated multilayer framework. Conversely, identifying vertices whose importance varies between layers and minimizing these differences requires considering the multilayer network as a whole. For example, identifying a location within a transportation network that is crucial for one public transit system but marginal for another can highlight specific points warranting focus for future development. Therefore, we initiate our discussion with a review of weight normalization, a technique employed to ensure that the scales of each layer are appropriately aligned. Subsequently, a novel metric, comprehensive evaluation, and graph algorithm are introduced.

\subsection{Weight Normalization}
In the context of multiplex networks, the normalization of edge weights constitutes a pivotal step preceding the computation of the Forman curvature. Two representative strategies were considered in this study. The first method is referred to as "mean normalization," in which the weights are rescaled so that their average equals one. The second method is referred to as "bounded scaling," in which the weights are mapped to a fixed interval, for example, $[1,10]$.

In the case of small graphs with relatively narrow weight ranges (e.g., spanning from $0.5$ to $5$), both methods yield similar curvature distributions. For instance, in Zachary's karate club graph, the histograms of curvature values under the two schemes are almost indistinguishable (Supplementary Figure \ref{karate club}). However, for networks with larger sizes and highly heterogeneous weights (e.g., ranging from $0.1$ to $1000$), the differences become pronounced. In a classical Erd\"{o}s--R\'{e}nyi random graph ($G(200, 0.5)$), mean normalization results in curvature values that are skewed over a broad interval (e.g., $[-500, 0]$), whereas bounded scaling produces a more compact distribution (e.g., $[-350, -50]$) (Supplementary Figure \ref{random graph}). The random graph $G(n,p)$ is defined by $n$ vertices, and each possible edge between pairs of vertices appears independently with a probability $p$.

Consequently, the selection of normalization should be congruent with the research objectives. In scenarios where the objective is to accentuate anomalous edges or identify bottlenecks within the network, mean normalization emerges as the optimal approach. Conversely, if the objective is to ensure stability and comparability across the entire graph, bounded scaling within a predefined interval, such as $[1,10]$, is more suitable. 

\begin{figure}[H]
     \centering
     \includegraphics[scale=0.4]{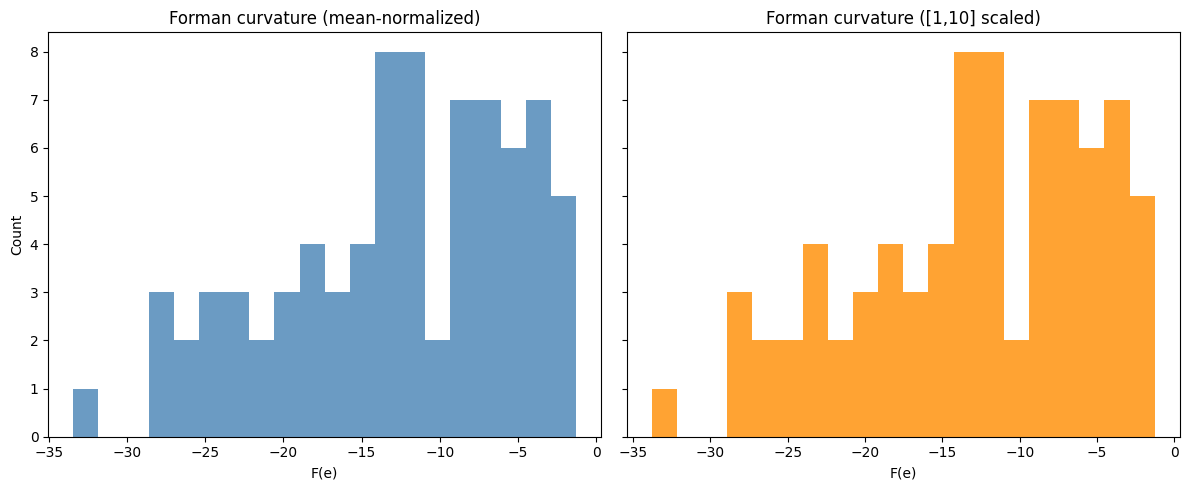} 
     \caption{The histogram of the Forman curvature on the Karate club graph.}
     \label{karate club}
\end{figure}

\begin{figure}[H]
     \centering
     \includegraphics[scale=0.4]{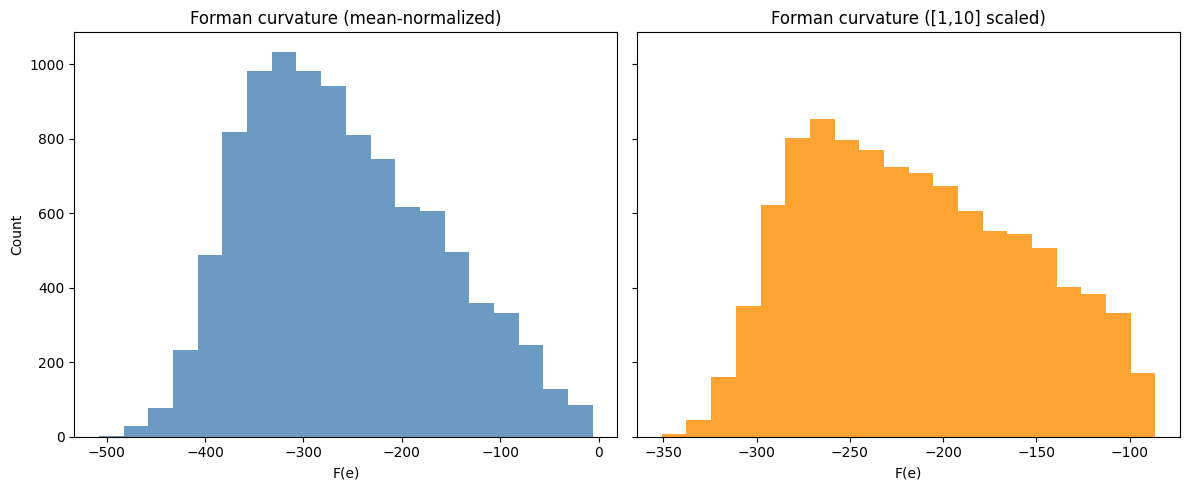} 
     \caption{The histogram of the Forman curvature on the random graph $G(200,0.5)$.}
     \label{random graph}
\end{figure}

\subsection{Graph Algorithm}
Before outlining the algorithm steps, we define a new indicator based on the inter-layer Forman curvature.
\begin{definition}
    Let $CG = (V_{CG}, E_{CG}, V, \mathcal{L}, w, m)$ be a compile graph. For a vertex $x \in V$, a \textit{comprehensive evaluation} of $x$ is defined as
    \begin{eqnarray*}
        CE(x) = \sum_{i,j \in \mathcal{L}: i<j} F((x^i,x^j)).
    \end{eqnarray*}
\end{definition}
\begin{remark}
\label{variance}
    For a vertex $x \in V$, consider the layers ordered such that $W(x_1) \leq W(x_2) \leq \cdots \leq W(x_L)$. Subsequently, by applying Theorem \ref{lower}, the lower bound for the comprehensive evaluation of $x$ is represented as follows:
    \begin{eqnarray*}
        CE(x) &=& \sum_{i,j \in \mathcal{L}:i<j} F((x^i,x^j))\\
        &\geq& -\sum_{i,j \in \mathcal{L}:i<j} (m(x^i)+m(x^j)) |\Gamma^{i,j}_{C,-} (x)|\max \left\{ \cfrac{W(x^i)}{W(x^l)} \mid x^l \in \Gamma^{i,j}_{C,-} (x)\right\}\\
        &=& -\sum_{i,j \in \mathcal{L}:i<j}(m(x^i)+m(x^j)) |\Gamma^{i,j}_{C,-} (x)| \cfrac{W(x^i)}{W(x^1)}.
    \end{eqnarray*}
    Alternatively, for the case of $W(x^1) = \cdots = W(x^L)$, the comprehensive evaluation of $x$ is denoted by $CE^\text{uni}(x)$ and calculated as
    \begin{eqnarray*}
        CE^\text{uni}(x) &=& \sum_{i,j \in \mathcal{L}} F((x^i,x^j))\\
        &=& - \sum_{i,j \in \mathcal{L}} (m(x^i) + m(x^j)).
    \end{eqnarray*}
    Therefore, for a given vertex $x \in V$, a larger difference between $CE(x)$ and $CE^\text{uni}(x)$ suggests a greater dispersion among the $W(x^i)$ values for the vertex. Note that while the variance and standard deviation can be used as indicators of variation, they might not be suitable for comparing the $W(x^i)$ dispersion across different vertices, as the representative value is different for each vertex.
\end{remark}

\begin{claim}\ 
Input a compile graph $CG = (V_{CG}, E_{CG}, V, \mathcal{L}, w, m)$.
\begin{enumerate}[label=\textbf{Step \arabic*.}]
    \item Perform weight normalization for each $G^l$ ($l \in [1, L]$).
    \item Calculate the difference between $CE(x) = \sum_{i<j} F((x^i,x^j))$ and $CE^\text{uni}(x)$ for each vertex.
    \item Compute $\sum_{y}F((x_0^i,y^i))$ for each layer $i$ of the vertex $x_0$ identified as having a large "Difference value" in Step 1.
    \item Compute $F((x_0^{i_0},y^{i_0}))$ for each intra-layer edge $e=(x_0^{i_0},y^{i_0})$ for a layer $i_0$ with identified in Step 2 as having a large value and extract an edge $e_0=(x_0^{i_0},y^{i_0})$ that exhibits a large curvature value.
\end{enumerate}
\end{claim}

First, Step 1 calculates the difference between $CE(x)$ and $CE^\text{uni}(x)$ for each vertex $x$. This allows for a comparison of the degree of variability among the $W(x^i)$ values for that vertex (as discussed in Remark \ref{variance}). Next, Step 2 examines a vertex $x_0$, which is identified as having a large difference value in Step 1. By comparing the sum of the incident intra-layer edge Forman curvatures for each layer $i$, this step helps distinguish layers where $x_0$ is important from those where it is not (see Remark \ref{Forman property}). Finally, Step 3 focuses on a specific layer $i_0$ where $x_0$ is not important  by comparing the Forman curvature of each edge.

\section{Simulation}
\subsection{Weakness Identification in Compile Graphs}
This section focuses on the primary objective of this study: weakness identification in compiled graphs. Using the Forman curvature and comprehensive evaluation, we demonstrated the ability to efficiently identify vulnerable points in complex multiplex graphs, thereby clarifying the practical significance of the proposed method. The base graph is a compiled graph constructed from three random graphs $G(n,p)$ with progressively modified values $p$. A random number between 0.01 and 1 is assigned as the weight of each vertex. Similarly, a random number between 1 and 10 is assigned as the weight of each edge. This weight assignment indicates that weight normalization was performed using bounded scaling for a dataset. These random graphs are then combined to construct a compile graph (Figure \ref{data_desc}).

\begin{example}
$CG = CG^{258}$ is the compile graph constructed from\newline $\left\{G(25,0.2), G(25,0.5), G(25,0.8)\right\}$.
\end{example}
\begin{figure}[H]
     \centering
     \includegraphics[scale=0.45]{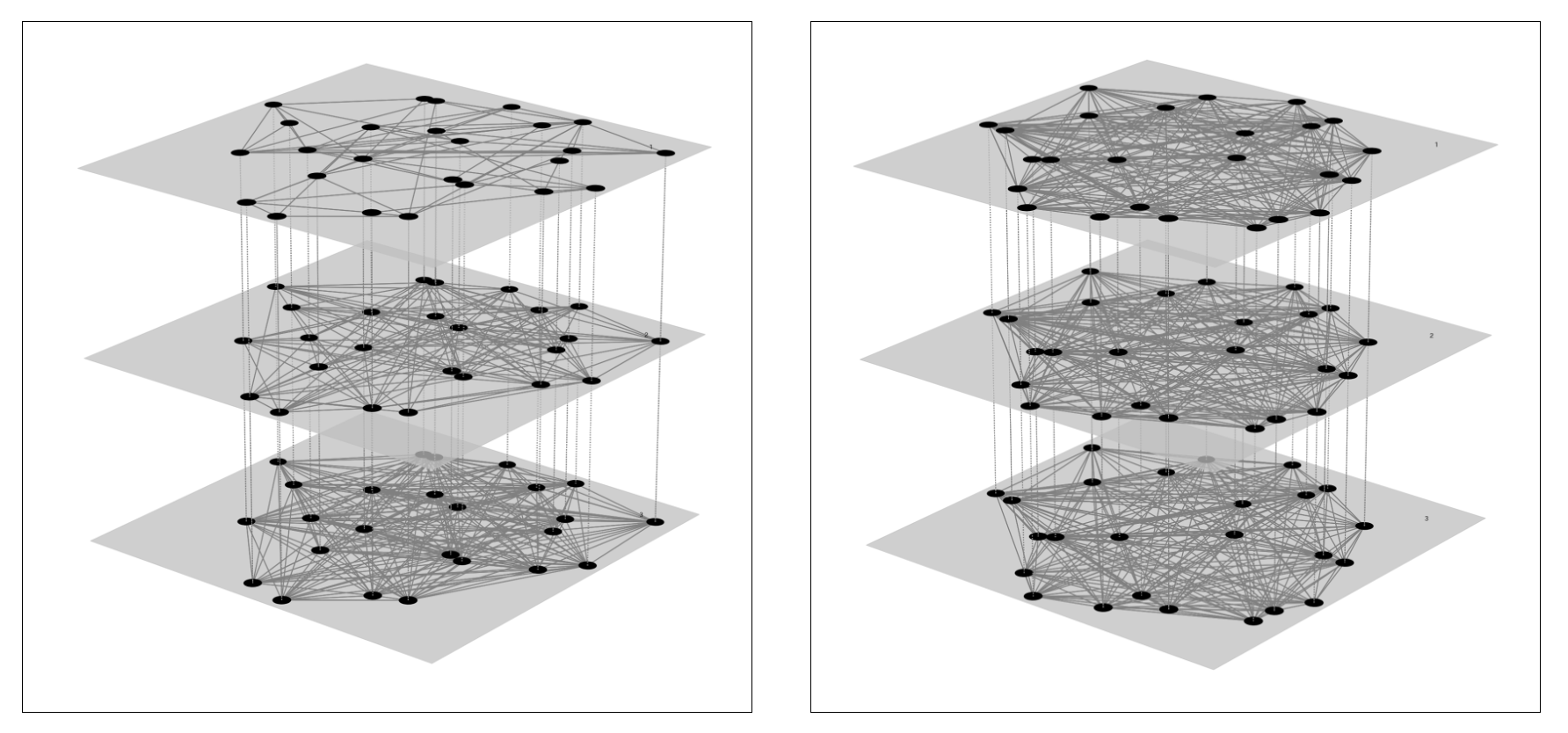} 
     \caption{\newline Left: $CG^{258}$ (Compile graph of $\left\{G(25,0.2), G(25,0.5), G(25,0.8)\right\}$).\newline Right: $CG^{888}$ (Compile graph of $\left\{G(25,0.8), G(25,0.8), G(25,0.8)\right\}$).}
     \label{data_desc}
\end{figure}
The results from applying this graphing algorithm show that Vertex 10 has the largest variation in its $W(x^i)$ values across the layers (Figure \ref{258step1}). Moreover, the sum of the intra-layer edge Forman curvatures for Vertex 10, when calculated for Layers 1, 2, and 3, is the largest in Layer 1. This leads to the conclusion that Vertex 10 is the least important in Layer 1 (Figure \ref{258step2}). Finally, within Layer 1, the Forman curvature for each edge between Vertices 10 and 20 has the largest curvature value (Figure \ref{258step3}).
 \begin{figure}[H]
     \centering
     \includegraphics[scale=0.45]{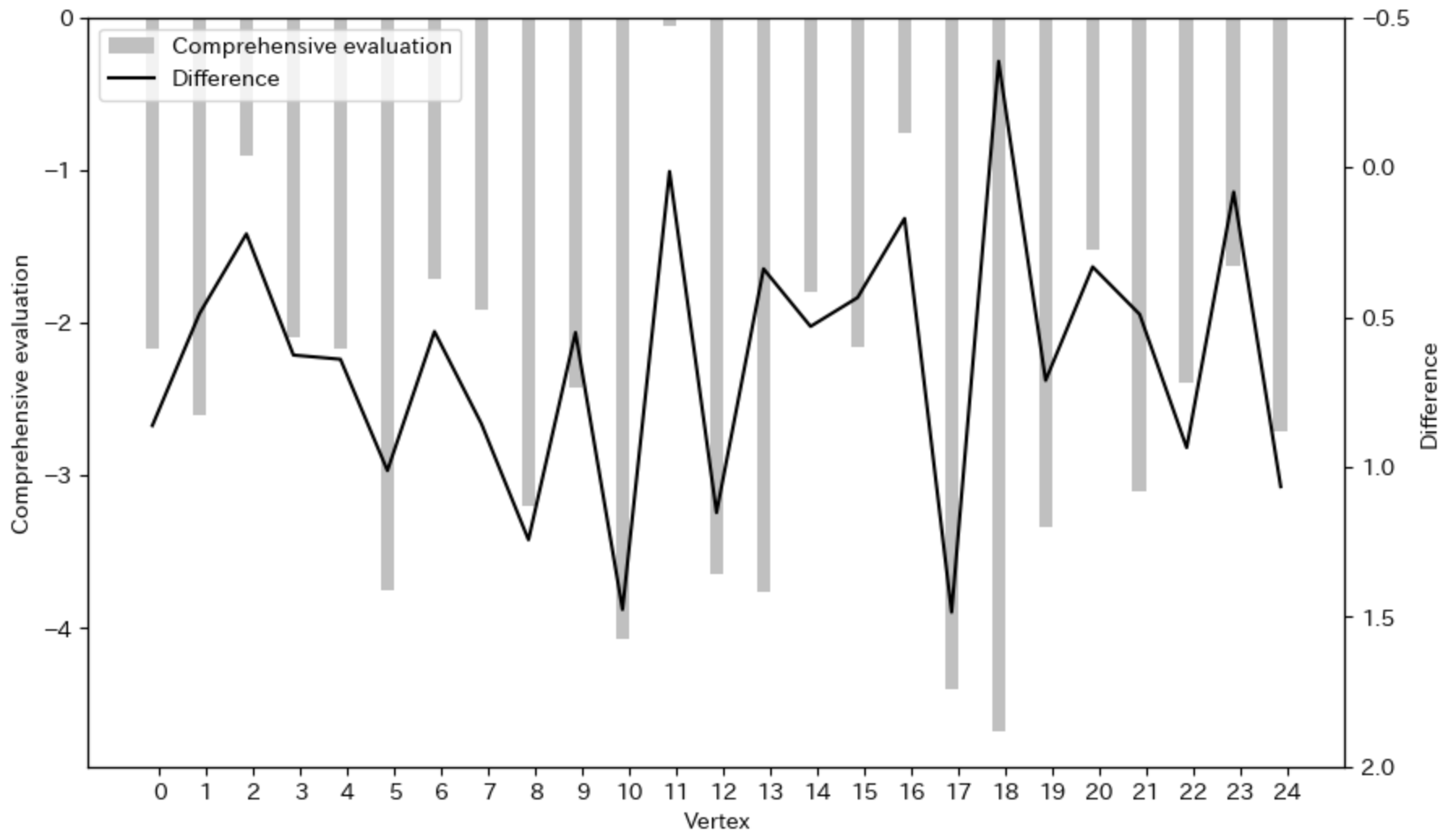} 
     \caption{Results of running Step 1 for $CG^{258}$. The "Difference" value represents the difference between $CE(x)$ and $CE^\text{uni}(x)$.}
     \label{258step1}
 \end{figure}

 \begin{figure}[H]
     \centering
     \includegraphics[scale=0.6]{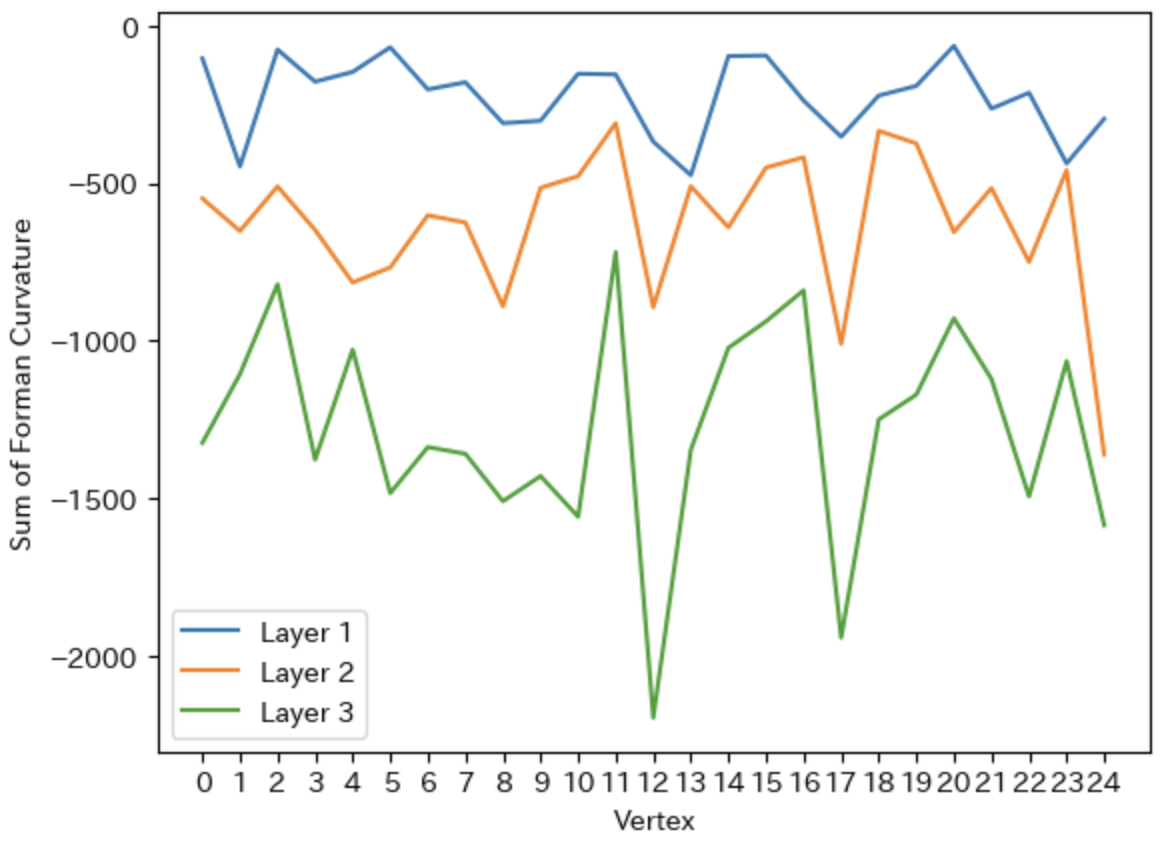} 
     \caption{Results of running Step 2 for $CG^{258}$.}
     \label{258step2}
 \end{figure}

 \begin{figure}[H]
     \centering
     \includegraphics[scale=0.4]{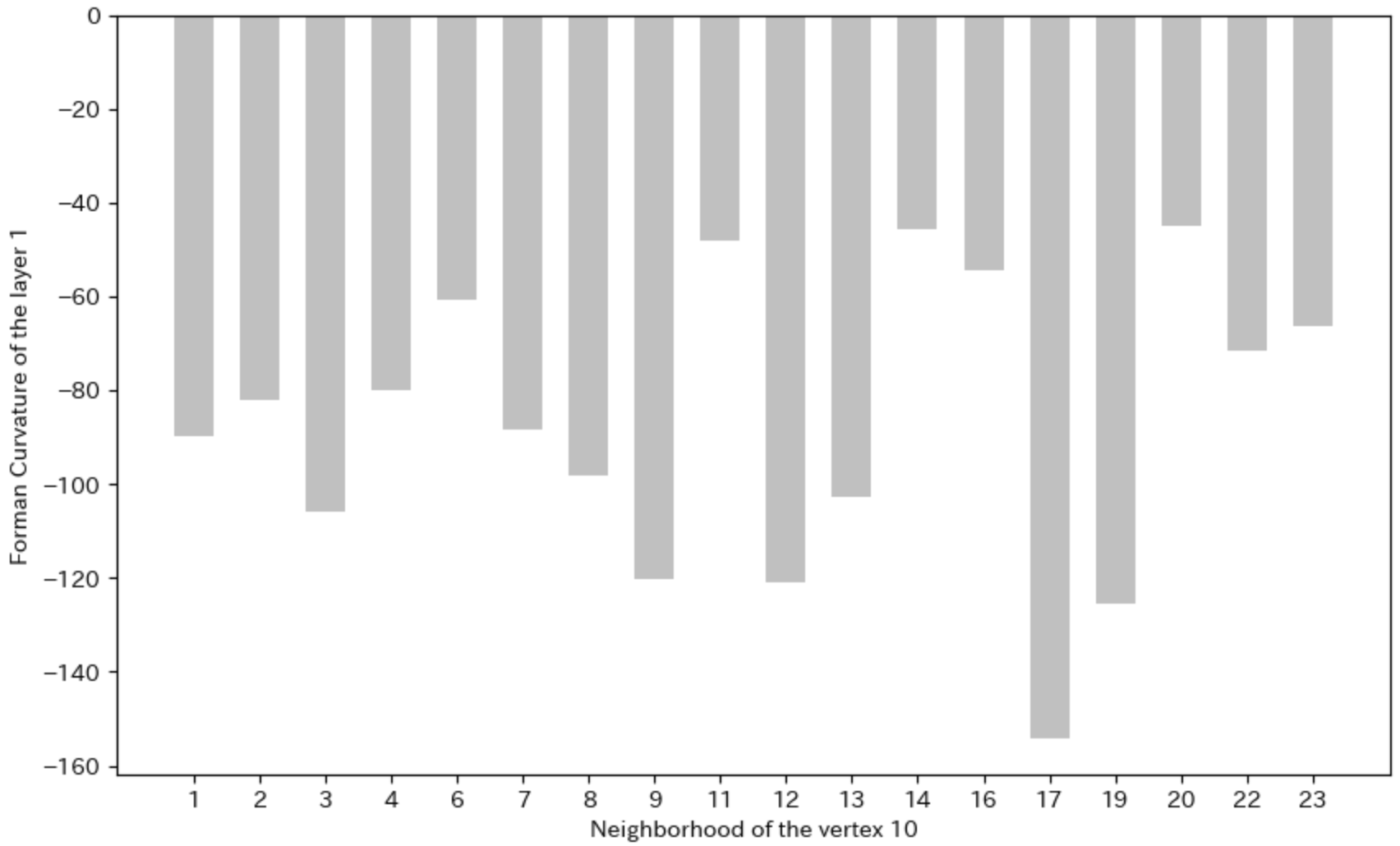} 
     \caption{Results of running Step 3 for Vertex 10 and Layer 1 in $CG^{258}$.}
     \label{258step3}
 \end{figure}

\begin{example}
$CG = CG^{888}$ is the compile graph constructed from\newline $\left\{G(25,0.8), G(25,0.8), G(25,0.8)\right\}$.
\end{example}
The results from this graphing algorithm indicate that Vertex 8 exhibits the largest $W(x^i)$ value across the layers (Figure \ref{888step1}). Additionally, the sum of the Forman curvatures of the intra-layer edges for Vertex 8 across Layers 1, 2, and 3 revealed that the largest value occurred in Layer 3. This leads to the conclusion that Vertex 10 is the least important in Layer 3 (Figure \ref{888step2}). Finally, within Layer 3, the calculation of the Forman curvature for each edge connected to Vertex 8 shows that the edge between Vertices 8 and 22 has the largest curvature value (Figure \ref{888step3}). 
 \begin{figure}[H]
     \centering
     \includegraphics[scale=0.45]{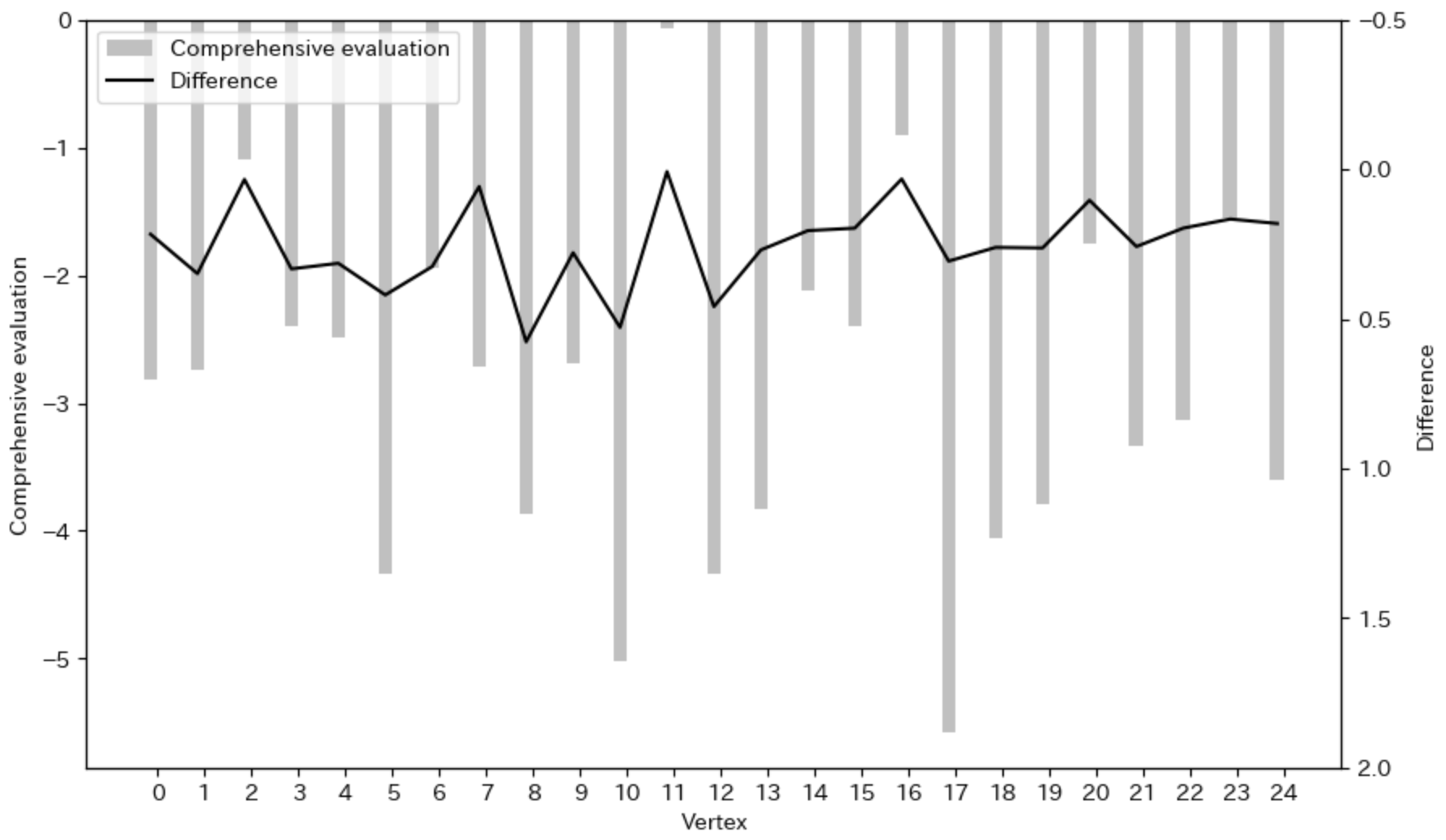} 
     \caption{Results of running Step 1 for $CG^{888}$.}
     \label{888step1}
 \end{figure}

 \begin{figure}[H]
     \centering
     \includegraphics[scale=0.6]{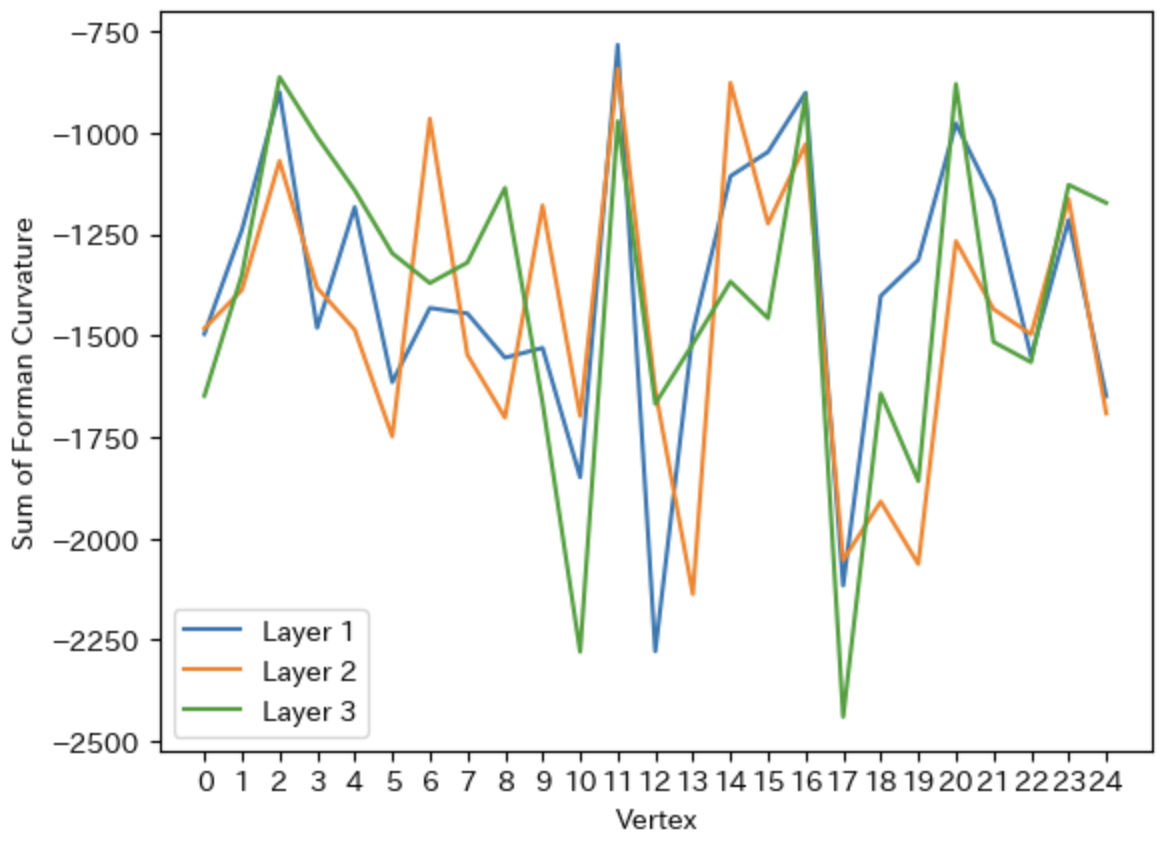} 
     \caption{Results of running Step 2 for $CG^{888}$.}
     \label{888step2}
 \end{figure}
 
 \begin{figure}[H]
     \centering
     \includegraphics[scale=0.5]{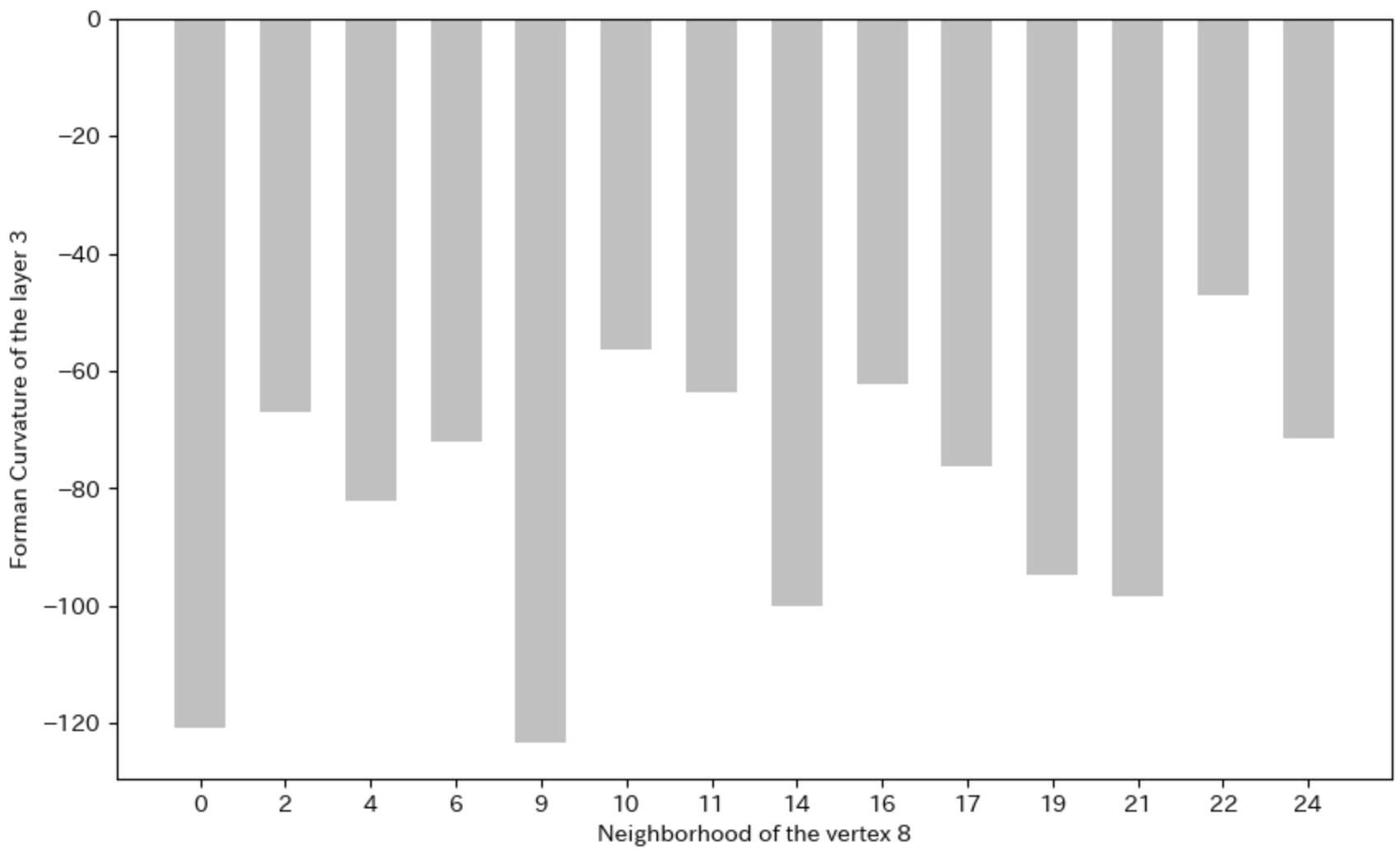} 
     \caption{Results of running Step 3 for Vertex 8 and Layer 3 in $CG^{888}$.}
     \label{888step3}
 \end{figure}

A comparison of the two examples shows that when the graph structure of each layer is similar, the overall "difference value" tends to be smaller for all vertices $x$, and the variance of these difference values across vertices is also smaller. Conversely, when the graph structures of the layers differ significantly, the "difference values" are generally larger, and their variance across vertices is also larger. Therefore, comparing these difference values is useful not only for clarifying the degree of variation in  $W(x^i)$ for each vertex, but also for understanding the graph structure of the entire multilayer network.

\subsection{Classification Performance Comparison}
In this section, we conduct comparative experiments based on the classification performance to further substantiate the effectiveness of the proposed method. By comparing it with existing metrics and methods for classification tasks, we provide supplementary validation of the comprehensive evaluation (CE) introduced in this study.
\subsubsection{Comparison with other graph metrics}
To evaluate the effectiveness of the proposed CE features based on Forman curvature, a comparison is made of their classification performance against traditional graph metrics such as degree, clustering coefficient, and betweenness centrality. To this end, we implemented a 5-fold cross-validation procedure employing a Random Forest classifier. Our analytical framework incorporated both accuracy (Acc) and macro-averaged F1 score (F1\_macro), with the objective of addressing potential class imbalances. The base graph is a compiled graph constructed from three Zachary's karate club graphs with progressively modified edge weights.

The results are summarized in Table \ref{tab:ce_vs_trad}. The utilization of CE features alone led to an accuracy of $0.4476 \pm 0.1227$ and an F1\_macro of $0.4278 \pm 0.1412$. In contrast, traditional graph metrics alone attained $0.4667 \pm 0.1292$ (Acc) and $0.4417 \pm 0.1481$ (F1\_macro). Notably, the integration of CE with conventional metrics resulted in a significant enhancement, yielding $0.5333 \pm 0.1292$ (Acc) and $0.4933 \pm 0.1501$ (F1\_macro). This finding suggests that the CE captures complementary information that is not present in conventional graph descriptors.

Feature importance analysis further corroborates this conclusion ( Table \ref{tab:feature_importance}). Among the top 10 features of the combined model, the CE features contributed significantly. For instance, CE\_stat\_3 and CE\_stat\_0 ranked second and third in importance, respectively, highlighting their relevance in discriminating between graph classes. While traditional metrics retained their significance, as evidenced by the top ranking of TRAD \_stat \_5, the integration of CE and traditional metrics significantly augmented the classifier's performance.

The findings indicate that incorporating the CE distribution of Forman curvature enhances the discriminatory power of standard graph metrics, thereby substantiating its value as a feature in graph classification.

\begin{table}[H]
\centering
\caption{Classification performance comparison between CE features, traditional graph metrics, and their combination using Random Forest (5-fold CV). Accuracy (Acc) and macro F1 score (F1\_macro) are reported as mean $\pm$ standard deviation.}
\label{tab:ce_vs_trad}
\begin{tabular}{lcc}
\hline
Features & Acc & F1\_macro \\
\hline
CE only       & 0.4476 $\pm$ 0.1227 & 0.4278 $\pm$ 0.1412 \\
Traditional only & 0.4667 $\pm$ 0.1292 & 0.4417 $\pm$ 0.1481 \\
Both (CE + Traditional) & 0.5333 $\pm$ 0.1292 & 0.4933 $\pm$ 0.1501 \\
\hline
\end{tabular}
\end{table}

\begin{table}[H]
\centering
\caption{Top 10 feature importance in the combined CE + traditional model.}
\label{tab:feature_importance}
\begin{tabular}{lc}
\hline
Feature & Importance \\
\hline
TRAD\_stat\_5 & 0.1492 \\
CE\_stat\_3   & 0.1263 \\
CE\_stat\_0   & 0.0776 \\
TRAD\_stat\_2 & 0.0764 \\
CE\_stat\_2   & 0.0762 \\
CE\_stat\_4   & 0.0758 \\
CE\_stat\_5   & 0.0752 \\
TRAD\_stat\_4 & 0.0655 \\
CE\_stat\_6   & 0.0607 \\
CE\_stat\_11  & 0.0601 \\
\hline
\end{tabular}
\end{table}

\subsubsection{Comparison with the Weisfeiler–Lehman}
To evaluate whether CE improves existing graph classification pipelines, we compared our approach with the Weisfeiler–Lehman (WL) subtree kernel. The experiments were conducted on synthetic graph datasets generated to contain class-distinguishing structural differences that are not easily captured by subtree patterns alone. Specifically, we considered two classes of graphs with the same degree distribution but differing higher-order connectivity and curvature-related properties (e.g., presence of bottlenecks, bridges, or clustering variations). Each dataset contained 300 graphs that were balanced across the two classes.
Both WL and CE were combined with a Random Forest classifier, and the performance was evaluated using repeated 5-fold cross-validation (20 repetitions, yielding 100 train/test splits). This design ensured that the results were not sensitive to a particular partitioning of data.

The results are summarized in Table \ref{tab:comparison with WL}. This indicates that while the WL kernel achieves strong classification performance on its own (average accuracy $\sim$0.78), curvature embeddings provide comparable or slightly higher performance (average accuracy $\sim$0.80). More importantly, the combination of WL and CE features yielded the best results overall, with accuracy and F1-scores further improved to $\sim$0.81–0.82 for the best-performing model.

A statistical comparison using the Wilcoxon signed-rank test confirmed that CE features significantly outperformed WL alone ($p < 0.01$) and that the combined model is not worse than either method individually. These findings suggest that curvature information encodes complementary structural properties that are not captured by WL subtree patterns, thereby enhancing classification pipelines when it is integrated with existing graph kernel methods.

\begin{table}[H]
\centering
\caption{Comparison of Weisfeiler–Lehman (WL), Comprehensive evaluation (CE), and their combination. Results are averaged over 20 repetitions of 5-fold cross-validation.}
\label{tab:comparison with WL}
\begin{tabular}{lcc}
\hline
Method & Accuracy & F1-macro \\
\hline
WL only        & $0.7763 \pm 0.0479$ & $0.8020 \pm 0.0419$ \\
CE only        & $0.7980 \pm 0.0395$ & $0.8025 \pm 0.0391$ \\
WL + CE        & $0.8123 \pm 0.0438$ & $0.8192 \pm 0.0438$ \\
\hline
\end{tabular}
\end{table}

\section{Conclusion}
This study introduced the Forman curvature for multilayer networks and demonstrated its utility. Specifically, we showed that computing the curvature of intra-layer edges allows for assessing vertex/edge importance within individual layers, whereas computing that of inter-layer edges provides insights into the relationships between layers and aspects of the overall multilayer network structure. The analysis of the intra-layer curvature confirms that the Forman curvature can play a similar role in assessing local importance within layers as in single-layer graphs. Concurrently, the analysis based on inter-layer curvature represents a novel application, offering a new method for using Forman curvature. In particular, the new metric, which is a comprehensive evaluation, introduced in this study achieves high accuracy in graph classification compared to other traditional metrics. Furthermore, it was demonstrated that integrating this approach with features derived from graph kernels results in enhanced accuracy.
Therefore, the Forman curvature-based indices and analysis framework introduced herein hold potential for application across a wide range of fields dealing with complex weighted multilayer systems. In future research, we plan to investigate the properties of the inter-layer Forman curvature in greater detail and quantify the degree of variation of $W(x^i)$ with greater precision.

\section*{Acknowledgment}
I would like to thank Prof. Hirohisa Yamada for providing numerous practical comments. I would also like to thank all the reviewers who helped improve this paper. 

I would like to thank Editage (www.editage.jp) for the English language editing.

\section*{Funding}
This work was supported in part by a donation from Aioi Nissay Dowa and JSPS KAKENHI [grant number 25K17253].
\section*{Conflict of interest}
None.

\section*{Declaration of generative AI and AI-assisted technologies in the manuscript preparation process.}
During the preparation of this work, the author used ChatGPT (OpenAI) in order to assist in organizing the manuscript structure and summarizing related work. After using this tool, the author reviewed and edited the content as needed and took full responsibility for the content of the published article.


\printbibliography

\end{document}